\documentclass[12pt]{amsart}

\usepackage{times,amsfonts,amsmath,amstext,amsbsy,amssymb,
amsopn,amsthm,upref,eucal}

\usepackage{xcolor}
\usepackage{verbatim}

\usepackage{url}

\usepackage{graphicx}

\newtheorem{theorem}{Theorem}

\newtheorem{proposition}{Proposition}
\newtheorem{corollary}{Corollary}

\newcommand{\classF}{\mathfrak{F} (m; H)}
\newcommand{\UxU}{U(\infty) \times U(\infty)}
\newcommand{\Mat}{\mathop{\mathrm{Mat}}\nolimits(\mathbb{N}, \mathbb{C})}
\newcommand{\classFNC}{\mathfrak{F} (m; \Mat)}
\newcommand{\tr}{\mathop{tr}\nolimits}
\title[Finiteness of Ergodic  Measures]{Finiteness of
Ergodic Unitarily Invariant Measures on Spaces of Infinite Matrices}
\author{Alexander I. Bufetov}
\address{The Steklov Institute of Mathematics, Moscow}
\address{The Institute for Information Transmission Problems, Moscow}  
\address{Rice University, Houston}
\date{12 August 2011}
\begin{document}

\begin{abstract}
The main result of this note, Theorem 2, is the following: a Borel measure on the space 
of infinite Hermitian matrices, that is invariant under the action of the infinite unitary group 
and that admits well-defined projections onto the quotient space 
of ``corners" of finite size, must be finite. A similar result, Theorem 1, is also established for 
unitarily invariant measures on the space of all infinite complex matrices. These results, 
combined with the ergodic decomposition theorem of [3], imply that the infinite Hua-Pickrell measures of 
Borodin and Olshanski [2] have finite ergodic components. 

The proof is based on the approach of Olshanski and Vershik [6]. First, it is shown  that 
if the sequence of orbital measures assigned to almost every point is weakly precompact, 
then our ergodic measure must indeed be finite.  The second step, which completes the proof, shows that 
if a unitarily-invariant measure admits well-defined projections onto the quotient space of finite corners, 
then for almost every point the corresponing sequence of orbital measures is indeed weakly precompact. 

\end{abstract}
\maketitle
\section{Introduction}\subsection{Statement of the main
results}\subsubsection{Unitarily invariant measures on spaces of infinite complex matrices} Let $\Mat$ be the
space of all infinite matrices whose  rows and columns are indexed by natural numbers and whose entries are
complex:
\[
\Mat \;=\; \left\{ z=(z_{ij})_{i,j \in \mathbb N}, \; z_{ij}\in \mathbb C \right\}.
\]

Let $U(\infty)$ be the infinite unitary group: an infinite matrix $u=(u_{ij})_{i,j\in {\mathbb N}}$ belongs to
$U(\infty)$ if there exists a natural number $n_0$ such that the matrix
$$
(u_{ij})_{i,j\in [1,n_0]}
$$
is unitary, while $u_{ii}=1$ if $i>n_0$ and $u_{ij}=0$ if $i\neq j$, $\max(i,j)>n_0$.

The group $U(\infty) \times U(\infty)$ acts on $\Mat$ by multiplication on both sides:
\[
T_{(u_1,u_2)}z \;=\; u_1zu_2^{-1}.
\]

Recall that a $U(\infty)\times U(\infty)$-invariant measure on $\Mat$, finite or infinite, is called
\textit{ergodic} if any $U(\infty)\times U(\infty)$-invariant Borel set either has measure zero or has complement
of measure zero. Finite ergodic $U(\infty)\times U(\infty)$-invariant measures on $\Mat$ have been classified by
Pickrell \cite{pickrell}. The first main result of this paper is that,
under natural assumptions, an ergodic $U(\infty)\times U(\infty)$-invariant measure on $\Mat$ must be finite.

Precisely, let $m\in\mathbb N$ and let $\classFNC$ denote the space of Borel measures $\nu$ on $\Mat$ such that for any $R>0$ we have
$$
\nu\left(\left\{z\in\Mat: \max_{i,j\leqslant m}|z_{ij}|<R\right\}\right)<+\infty.
$$
\begin{theorem} \label{mainz}
If a $\UxU$-invariant Borel measure from the class $\classFNC$ is ergodic then it is finite.
\end{theorem}

A measure  $\nu\in\classFNC$ is automatically sigma-finite,
clearly satisfies all assumptions of the ergodic decomposition theorem of \cite{Bufetov} and therefore admits a
decomposition into ergodic components. By definition, almost all ergodic components of  a 
measure $\nu\in\classFNC$ must themselves lie in the class $\classFNC$. 
Let ${\mathfrak M}_{erg}(\Mat)$
stand for the set of $\UxU$-invariant ergodic Borel probability measures on $\Mat$; the set ${\mathfrak
M}_{erg}(\Mat)$ is a Borel subset of the space of all Borel probability measures on $\Mat$ (see, e.g.,
\cite{Bufetov}, where the claim is proved for all measurable Borel actions of inductively compact groups). Theorem
\ref{mainz} and the ergodic decomposition theorem of \cite{Bufetov} now implies the following
\begin{corollary}\label{corz}
For any $\UxU$-invariant Borel measure 
$$
\nu\in \classFNC
$$ 
there exists a unique sigma-finite Borel measure
${\tilde\nu}$ on ${\mathfrak M}_{erg}(\Mat)$ such that
\begin{equation}
\label{ergdecfmz} \nu=\int\limits_{{\mathfrak M}_{erg}(\Mat)}\eta d{\tilde \nu}(\eta).
\end{equation}
\end{corollary}

The integral in \eqref{ergdecfmz} is understood in the usual weak sense: for every Borel subset $A\subset\Mat$ we
have
$$
\nu(A)=\int\limits_{{\mathfrak M}_{erg}(\Mat)}\eta(A) d{\tilde \nu}(\eta).
$$

\subsubsection{Unitarily invariant measures on spaces of infinite
Hermitian matrices}

Now let $H\subset \Mat$ be the space of infinite Hermitian matrices:
$$
H=\{h=(h_{ij})_{i,j\in {\mathbb N}}, h_{ij}={\overline  h_{ji}}\}.
$$
The group  $U(\infty)$ naturally acts on the space $H$ by conjugation. Finite ergodic $U(\infty)$-invariant
measures on $H$ have also been classified by Pickrell \cite{pickrell} (see also Olshanski and Vershik \cite{OV}). 
An analogue of Theorem \ref{mainz} holds
in this case as well.

Precisely, a Borel  measure $\nu$ on $H$ is said to belong to the class $\mathfrak F(m, H)$ if for any $R>0$ we
have
$$
\nu(\{h\in H: \max\limits_{i\leq m,j\leq m} |h_{ij}|\leq R\})<\infty.
$$
\begin{theorem}\label{mainh}
If a  $U(\infty)$-invariant measure from the class $\mathfrak F(m, H)$ is ergodic, then it is finite.
\end{theorem}
As before, let ${\mathfrak M}_{erg}(H)$ stand for the set of $U(\infty)$-invariant ergodic Borel probability
measures on $H$; the set ${\mathfrak M}_{erg}(H)$ is a Borel subset of the space of all Borel probability measures
on $H$. Theorem \ref{mainh} now implies
\begin{corollary}\label{corh}
For any $U(\infty)$-invariant Borel measure $\nu\in \mathfrak F(m, H)$ there exists a unique sigma-finite Borel
measure ${\tilde\nu}$ on ${\mathfrak M}_{erg}(H)$ such that
\begin{equation}
\label{ergdecfmh} \nu=\int\limits_{{\mathfrak M}_{erg}(H)}\eta d{\tilde \nu}(\eta).
\end{equation}
\end{corollary}

The integral in \eqref{ergdecfmh} is again understood in the  weak sense.

One expects similar results to hold for all the $10$ series of homogeneous spaces (see, e,.g., \cite{O1, O2}).

\subsubsection{Infinite Hua-Pickrell measures}

A natural example of measures lying in the class ${\mathfrak F}(m,H)$ is given by infinite Hua-Pickrell measures
introduced by Borodin and Olshanski \cite{BO}, Section 8, Subsection ``Infinite measures''. In fact, for any $m\in
{\mathbb N}$, Borodin and Olshanski give explicit examples of measures lying in the class ${\mathfrak F}(m,H)$ but not in
the class ${\mathfrak F}(m-1,H)$. Starting from the Pickrell measures \cite{pickrell3}, a similar construction can
be carried out to obtain infinite $\UxU$-invariant measures on $\Mat$ lying in the class ${\mathfrak F}(m,\Mat)$
but not in the class ${\mathfrak F}(m-1,\Mat)$ for any $m\in {\mathbb N}$. Corollaries \ref{corz}, \ref{corh} show
now that ergodic components of infinite Hua-Pickrell measures are finite.

\subsection{Outline of the proofs of Theorems \ref{mainz}, \ref{mainh}.}

Olshanski and Vershik \cite{OV} gave a completely different proof for Pickrell's Classification Theorem of 
$U(\infty)$-invariant ergodic measures on $H$, and their method has been adapted to 
ergodic $\UxU$-invariant measures on $\Mat$ by Rabaoui \cite{rabaoui1}, \cite{rabaoui2}.
The proof of Theorems \ref{mainz}, \ref{mainh} is  based on the Olshanski-Vershik approach.

First, following Vershik \cite{V1}, to each infinite matrix we assign its sequence of {\it orbital measures}
obtained by averaging over exhausting sequences of compact subgroups in our infinite-dimensional unitary groups. A
simple general argument shows that precompactness of the family of orbital measures for almost all points implies
finiteness of an ergodic measure. Using the work of Olshanski and Vershik \cite{OV} and Rabaoui
\cite{rabaoui1}, \cite{rabaoui2}, we give a sufficient condition, called ``radial boundedness'' of a matrix, for weak
precompactness of its family of orbital measures: namely, it is shown that the sequence of orbital measures is
weakly precompact as soon as the norms   (and, in case of $H$, also the traces) 
of  $n\times n$ ``corners'' of our matrix do not grow too fast as
$n\to\infty$. To complete the proof of Theorem \ref{mainh}, it remains to show that with respect to any measure in the class ${\mathfrak F}(m, H)$, almost all
matrices are indeed radially bounded (the same statement, with the same proof, also holds for $\classFNC$). This
is done in two steps: first, it is shown that if a measure from the class ${\mathfrak F}(m, H)$ is
$U(\infty)$-invariant, then its suitably averaged conditional measures yield a {\it finite} $U(\infty)$-invariant
measure --- with respect to which almost all points must then be radially bounded; second, applying a finite
permutation of columns and rows, one deduces radial boundedness for the initial matrix and completes the proof.

\subsection{Projections and conditional measures}

For $n\in \mathbb N$, let ${\rm Mat}(n, {\mathbb C})$ be the space of all $n\times n$ complex matrices.

Introduce a map 
\[
\textstyle \Pi_{[1,n]}: \Mat\rightarrow {\rm Mat}(n, {\mathbb C})
\]
by the formula
\[
\textstyle \Pi_{[1,n]}z \;=\; (z_{ij})_{i,j=1,\ldots,n}, \ z \in \Mat.
\]
If a measure $\nu$ on $\Mat$ is infinite, then the projection $\left(\Pi_{[1,n]}\right)_*\nu$ may fail to be
well-defined. The class $\classFNC$ consists precisely of those measures $\nu$ for which the projection
$\left(\Pi_{[1,m]}\right)_*\nu$ (and, consequently, all projections $\left(\Pi_{[1,n]}\right)_*\nu$ for
$n\geqslant m$) are indeed well-defined. Equivalently, by Rohlin's Theorem on existence of conditional measures, a
measure $\nu$ belongs to the class $\classFNC$ if and only if:
\begin{enumerate}
\item there exists a measure
$\overline{\nu}$ on the space ${\rm Mat}(m, {\mathbb C})$ assigning finite weight
to every compact set;
\item for
$\overline{\nu}$-almost every $z^{(m)}\in {\rm Mat}(m, {\mathbb C})$ there exists a  Borel probability measure
$\nu_{z^{(m)}}$ on $\Mat$ supported on the set $\left(\Pi_{[1,m]}\right)^{-1}z^{(m)}$ such that for every Borel
subset $A\subset\Mat$ the map
$$
z^{(m)}\to\nu_{z^{(m)}}(A)
$$
is $\overline{\nu}$-measurable and that we have a decomposition
\begin{equation}
\label{ergdecz}  \nu=\int\limits_{{\rm Mat}(m,\mathbb{C})} \nu_{z^{(m)}}\,d\overline{\nu}\left(z^{(m)}\right)
\end{equation}
again understood in the weak sense.
\end{enumerate}
A similar description can be given for measures in the class $\classF$: a Borel measure $\nu$ on $H$ belongs to
the class $\classF$ if and only if there exists a measure $\overline{\nu}$ on the space $H(m)$ of $m\times
m$-Hermitian matrices which assigns finite weight to every compact set and, for $\overline{\nu}$-almost every
$h^{(m)}\in H(m)$ there exists a Borel probability measure $\nu_{h^{(m)}}$ such that
\begin{equation}
\label{ergdech} \nu=\int_{H(m)}\nu_{h^{(m)}}\,d\overline{\nu}\left(h^{(m)}\right),
\end{equation}
where the decomposition (\ref{ergdech}) is understood in the same way as the decomposition (\ref{ergdecz}).


{\bf Acknowledgements.} Grigori Olshanski posed the problem to me, and I am greatly indebted to him. I am deeply
grateful to Sevak Mkrtchyan and Konstantin Tolmachov for helpful discussions. I am deeply grateful to Lisa
Rebrova, Nikita Kozin and Nikita Medyankin for typesetting parts of the manuscript.

Part of this work was done when I was visiting Victoria University Wellington, the Joint Institute for
Nuclear Research in Dubna and Kungliga Tekniska H{\"o}gskolan in Stockholm. I am deeply grateful to these institutions 
for their warm hospitality.

This work was supported in part by an Alfred P. Sloan Research Fellowship, by the Grant MK-4893.2010.1 of the
President of the Russian Federation, by the Programme on Mathematical Control Theory of the Presidium of the
Russian Academy of Sciences, by the Programme 2.1.1/5328 of the Russian Ministry of Education and Research, by the
RFBR-CNRS grant 10-01-93115,  by the Edgar Odell Lovett Fund at Rice University and by the National Science
Foundation under grant DMS~0604386.

\section{Weak recurrence}
The proof is based on the following simple general observation. Let $X$ be a complete metric space, and let $G$ be
an inductively compact group, in other words,
$$
G=\bigcup\limits_{n=1}^{\infty}K(n), \ \ K(n)\subset K(n+1)
$$
where the groups $K(n)$, $n\in\mathbb{N}$, are compact and metrizable. Let $\mathfrak{T}$ be a continuous action
of $G$ on $X$ (continuity is here understood with respect to the totality of the variables).

Each group $K(n)$ is endowed with the Haar measure $\mu_{K(n)}$, and to each point $x\in X$ we assign, following
Vershik \cite{V1}, the corresponding sequence of \emph{orbital measures} $\mu_{K(n)}^x$ on $X$ given by the
formula
$$
\int_X f(y)\,d\mu_{K(n)}^x(y)=\int_{K(n)}f\left(T_gx\right)\mu_{K(n)}(g),
$$
valid for any bounded continuous function $f$ on $X$. Given a family ${\mathfrak A}$  of Borel probability
measures on $X$, we say that the family $\mathfrak A$ is \emph{weakly recurrent} if for any positive bounded
continuous function $f$ on $X$ we have
$$
\inf_{\nu\in {\mathfrak A}}\int f\,d\nu>0.
$$

\begin{proposition}
\label{weakrecimplfin} Let $\nu$ be an ergodic $\mathfrak{T}$-invariant measure on $X$ that assigns finite weight
to every ball and admits a set $B$, $\nu(B)>0$, such that for every $x\in B$ the sequence of orbital measures
$\mu_{K(n)}^x$ is weakly recurrent. Then $\nu$ is finite.
\end{proposition}

\begin{proof} Consider the space $L_2(X, \nu)$; for $n\in\mathbb{N}$,
let $L_2(X, \nu)^{K(n)}$ be the subspace of $K(n)$-invariant functions, and let $P_n:L_2(X, \nu)\to L_2(X,
\nu)^{K(n)}$ be the corresponding orthogonal projection.

If the measure $\nu$ is ergodic and infinite, then
\begin{equation}
\label{emptyinter} \bigcap\limits_{n=1}^{\infty} L_2(X,\nu)^{K(n)}=0.
\end{equation}

Indeed, let $L_2(X, \nu)^{G}$ be the subspace of $G$-invariant square-integrable functions. By definition, we have
\begin{equation}
\bigcap\limits_{n=1}^{\infty} L_2(X,\nu)^{K(n)}=L_2(X, \nu)^{G}.
\end{equation}

Now, if the measure $\nu$ is ergodic and assigns finite weight to every ball, then, by results of \cite{Bufetov},
it is also indecomposable in the sense that any Borel set $A\subset X$ such that for any $g\in G$ we have
$\nu(T_gA\Delta A)=0$ must satisfy either $\nu(A)=0$ or $\nu(X\setminus A)=0$. It follows that $L_2(X,
\nu)^{G}=0$, and \eqref{emptyinter} is proved.

For any $f\in L_2(X,\nu)$ we thus have $P_nf\to 0$ in $L_2(X,\nu)$ as $n\to\infty$. Along a subsequence we then
also have $P_{n_k}f\to 0$ almost surely with the respect to the measure $\nu$.

If $f$ is continuous and square-integrable, then the equality
$$
P_n f(x)=\int\limits_X f(y)\,d\mu_{K(n)}^x(y)
$$
holds for $\nu$-almost all $x$.

Take, therefore, $f$ to be a positive, continuous, square-integrable function on $X$ (the existence of such a
function follows from the fact that the measure $\nu$ assigns finite weight to balls: indeed, taking $x_0\in X$,
letting $d$ be the distance on $X$, and setting $f(x)=\psi(d(x_0,x))$, where $\psi:\mathbb{R}\to\mathbb{R}$ is
positive, continuous, and decaying rapidly enough at infinity, we obtain the desired function).

If $\nu$ is ergodic and infinite, then, from the above, for almost all $x\in X$ we have
$$
\lim\limits_{n\to\infty}\int f\,d\mu_{K(n)}^x=0.
$$
In particular, for $\nu$-almost all $x\in X$, the sequence of orbital measures is not weakly recurrent, which
contradicts the assumptions of the proposition.
\end{proof}

{\bf Remark.} The argument above, combined with the ergodic decomposition theorem of \cite{Bufetov}, yields a slightly
stronger statement:  if a  $\mathfrak{T}$-invariant measure $\nu$  on $X$ that
assigns finite weight to every ball is such that for $\nu$-almost every every $x\in X$ the sequence of orbital
measures $\mu_{K(n)}^x$ is weakly recurrent, then the ergodic components of $\nu$ are almost surely finite.

 It remains to derive
Theorems \ref{mainz}, \ref{mainh} from Proposition \ref{weakrecimplfin}. We start with Theorem \ref{mainh}.

\section{Proof of Theorem \ref{mainh}}
\subsection{Radial boundedness}

A matrix $h\in H$ will be called \emph{radially bounded} in $H$ if
$$
\sup_{n\in\mathbb{N}}\frac{|\tr\left(\Pi_{[1,n]}h\right)|}{n}<+\infty, \ 
\sup_{n\in\mathbb{N}}\frac{\tr\left(\Pi_{[1,n]}h\right)^2}{n^2}<+\infty.
$$
We shall now see that if $h\in H$ is radially bounded in $H$, then the family of orbital measures $\mu_n^h$, $n\in
{\mathbb N}$,  is precompact in the weak topology on $H$, and, consequently, weakly recurrent.

Recall that if $X$ is a complete separable metric space, $\mathfrak M(X)$  the space of Borel probability measures
on $X$, then the {\it weak topology} on ${\mathfrak M}(X)$ is defined as follows. Let $ f_1,\ldots, f_k: \; X
\longrightarrow \mathbb R $ be bounded continuous functions on $X$, let $ \varepsilon_1,\ldots, \varepsilon_k
\;>\;0, $ let $\nu_0 \in \mathfrak M(X)$ and consider the set
\begin{equation} \label{bwt}
\left\{ \nu \in \mathfrak M(X): \; \left| \int f_i \, d\nu - \int f_i \, d \nu_0 \right| < \varepsilon_i, \; i=1,
\ldots, k\right\}
\end{equation}
Sets of the form \eqref{bwt} form the basis of the weak topology on $\mathfrak M(X)$. Our assumptions on $X$ imply
that the space $\mathfrak M(X)$ endowed with the weak topology is itself metrizable and separable; for instance,
the L\'evy-Prohorov metric or the Kantorovich-Rubinstein metric induce the weak topology on $X$ (see, e.g.,
\cite{Bogachev}, Section 8.3). The symbol $\Rightarrow$ will denote weak convergence in the space $\mathfrak
M(X)$.

It is clear that weak precompactness of a family of probability measures implies  weak recurrence.

\begin{proposition}
\label{radbddprecomph} If a matrix $h\in H$ is radially bounded then the sequence
$\left\{\mu_n^h\right\}_{n\in\mathbb{N}}$ of orbital measures corresponding to $h$ is weakly precompact.
\end{proposition}

This Proposition is an immediate Corollary of Theorem 4.1 in Olshanski-Vershik \cite{OV}. Indeed, let $h\in H$ be
radially bounded, let
\[
h(n) = \Pi_{[1,n]}h=(h_{ij})_{i,j=1,\ldots,n},
\]
let
\[
\lambda_1^{(n)} \ge \ldots \ge \lambda_{k_n}^{(n)} \ge 0
\]
be the nonnegative eigenvalues of $h(n)$ arranged in decreasing order, and let
\[
\widetilde\lambda_1^{(n)} \le \widetilde\lambda_2^{(n)} \le \ldots \le \widetilde\lambda_{l_n}^{(n)} <0
\]
be the negative eigenvalues of $h(n)$ arranged in increasing order. Set
\[
x_i^{(n)} = \frac{\lambda_i^{(n)}}{n}, \qquad \widetilde x_i^{(n)} = \frac{\widetilde \lambda_i^{(n)}}{n};
\]
\[
\gamma_1^{(n)} = \frac{\tr h(n)}{n}, \qquad \gamma_2^{(n)} = \frac{\tr h^2(n)}{n^2}.
\]
Let $h$ is radially bounded, and let positive constants $C_1, C_2$ be such that for all $n\in {\mathbb N}$ we have
$$
{|\tr\left(\Pi_{[1,n]}h\right)|}\leq C_1 n, \ \  
\tr\left(\Pi_{[1,n]}h\right)^2\leq C_2n^2.
$$
We clearly have 
$$
|\gamma_1^{(n)}|\leq C_1, \ \ 0\leq \gamma_2^{(n)}\leq C_2,
$$
and, for all $i= 1, \dots, n$, we have 
$$
|x_i^{(n)}|, \ |{\tilde x}_i^{(n)}|  \leq C_2.
$$
Therefore, any infinite set of natural numbers contains a subsequence $n_r$ such that
sequences $\gamma_1^{(n_r)}, \gamma_2^{(n_r)}$, as well as the sequences $x_i^{(n_r)}, \widetilde x_i^{(n_r)}$ for
all $i=1,2,\ldots$ converge to a finite limit as $r \to \infty$. By the Olshanski-Vershik Theorem (Theorem 4.1 in
\cite{OV}), in this case the sequence $\mu_{n_r}^h$ of orbital measures weakly converges (in fact, to an ergodic
$U(\infty)$-invariant probability measure) as $r \to \infty$. The Proposition is proved completely.

{\bf Remark.} The converse claim (which, however, we do not need for our argument) also holds: if the sequence 
of orbital measures for a matrix $h\in H$ is weakly pecompact, then the matrix $h$ is radially bounded.
 This immediately follows from claim
(ii) of Theorem 4.1 of Olshanski and Vershik \cite{OV}. Note that, while claim (ii) in \cite{OV} is only
formulated for the full sequence of orbital measures, the same result, with the identical proof, is valid for any
infinite subsequence of orbital measures.

Observe that Theorem 4.1 in  Olshanski-Vershik \cite{OV} as well as the Ergodic Decomposition Theorem of
Borodin-Olshanski \cite{BO} immediately imply the following
\begin{proposition}
\label{finradbdd} If $\nu$ is a finite Borel $U(\infty)$-invariant measure on $H$, then $\nu$-almost every $h\in
H$ is radially bounded.
\end{proposition}
\begin{proof}
Indeed, if $\nu$ is an ergodic probability measure, then the claim is part of the statement of the
Olshanski-Vershik Theorem: in this case, for $\nu$-almost all $h\in H$, the sequence of orbital measures $\mu_n^h$
weakly converges to $\nu$. For a general finite measure, the result follows from the Ergodic Decomposition Theorem
of Borodin and Olshanski \cite{BO}.
\end{proof}

To complete the proof of Theorem \ref{mainh}, it remains to establish
\begin{proposition}\label{radbdd}
If a $U(\infty)$-invariant measure $\nu$ belongs to the class $\classF$ for some $m\in\mathbb{N}$, then $\nu$-almost every $h\in H$ is radially bounded.
\end{proposition}

\subsection{Proof of Proposition \ref{radbdd}}
For a matrix $z\in\Mat$, $n\in {\mathbb N}$, denote
$$
\Pi_{[n,\infty)}z=(z_{ij})_{i,j=n, n+1, \dots}.
$$

We start by showing that, under the assumptions of the proposition, 
for $\nu$-almost every $h\in H$ the matrix $\Pi_{[m, \infty)}h$ is radially bounded.

Take a measure $\nu \in \mathfrak F(m; H)$ and consider the corresponding canonical decomposition (\ref{ergdech})
into conditional measures.

\begin{proposition}
\label{uinv} Let $\nu \in \mathfrak F(m; H)$ be $U(\infty)$-invariant. Then for ${\overline \nu}$-almost every
$h^{(m)} \in H(m)$ the probability measure
\[
\textstyle \left( \Pi_{[m,\infty)} \right)_* \nu_{h^{(m)}}
\]
on $H$ is also $U(\infty)$-invariant.
\end{proposition}

\begin{proof}

Let $U_m(\infty) \subset U(\infty)$ be the subgroup of matrices $u = (u_{ij})$ satisfying the conditions:

\begin{enumerate}
\item if $\min (i,j) \le m, \; i\ne j$, then $u_{ij}=0$
\item if $i \le m$, then $u_{ii}=1$
\end{enumerate}

It follows from the definitions that if $u \in U_m(\infty)$, then $\Pi_{[m,\infty)}u \in U(\infty)$, and that the
map
\[
\Pi_{[m, \infty)}: \, U_m(\infty) \longrightarrow U(\infty)
\]
is a group isomorphism.

For $u \in U(\infty)$ let ${\mathfrak t}_u: \, H\to H$ be given by the formula
\[
\mathfrak t_u (h) \;=\; u^{-1}hu.
\]
Let $U'_m(\infty) \subset U_m(\infty)$ be a countable subgroup such that any Borel probability measure $\eta$ on
$H$ satisfying $ (\mathfrak t_u)_*\eta = \eta$ for all $u \in U'_m(\infty)$ must be invariant under the whole
group $U_m(\infty)$.

Uniqueness of Rohlin's system of conditional measures implies that for $\overline \nu$-almost every $h^{(m)} \in
H(m)$ and every $u \in U'_m(\infty)$ we have
\begin{equation} \label{uuinv}
\nu_{h^{(m)}} \;=\; (\mathfrak t_u)_*\nu_{h^{(m)}}.
\end{equation}
By definition of the subgroup $U'_m(\infty)$, the equality \eqref{uinv} also holds for all $u \in U_m(\infty)$.
Now let $A$ be a measurable subset of $H$, and let
\[
\textstyle \widetilde A_{h^{(m)}} \;=\; \left\{ h \in H: \, \Pi_{[1,m]} h=h^{(m)}, \, \Pi_{[m,\infty)}h \in A
\right\}.
\]

Let $u \in U(\infty)$ and let $\widetilde u \in U_m(\infty)$ be defined by the formula
\[
\textstyle \Pi_{[m,\infty)} \widetilde u \;=\; u.
\]
 From the definitions it follows:
\[
\mathfrak t_{\widetilde u} (\widetilde A_{h^{(m)}}) \;=\; \textstyle \left\{ h \in H: \, \Pi_{[1,m]} h = h^{(m)},
\, \Pi_{[m,\infty)} h \in \mathfrak t_u(A). \right\}
\]
Since
\[
\nu_{h^{(m)}} ( \widetilde A_{h^{(m)}} ) \; = \; \nu_{h^{(m)}} ( \mathfrak t_{\widetilde u} (\widetilde
A_{h^{(m)}} ) ),
\]
we have
\[
\textstyle \left( \Pi_{[m, \infty)} \right)_* \nu_{h^{(m)}} (A) \;=\; \left( \Pi_{[m, \infty)} \right)_*
\nu_{h^{(m)}} \left( \mathfrak t_u (A) \right),
\]
and the proposition is proved.
\end{proof}


\begin{corollary}
\label{projradbdd} If $\nu \in \mathfrak F(m; H)$ is $U(\infty)$-invariant, then for $\nu$-almost every $h\in H$
the the matrix $\Pi_{[m,\infty)} (h)$ is radially bounded.
\end{corollary}

We proceed with the proof of Proposition \ref{radbdd}.
 Let $\check u \in U(\infty)$ be defined as follows:
\begin{align}
\label{uhatdef}
\check u_{i,m+i} = \check u_{m+i,i} = 1 & \qquad i=1,\ldots, m\\
\check u_{2m+i,2m+i}=1 & \qquad i\in \mathbb N\\
\check u_{ij} = 0 & \qquad\text{otherwise}.
\end{align}

\begin{proposition}\label{radest}
Let $h \in H$. If $\Pi_{[m,\infty)}(h)$ and $\Pi_{[m,\infty)}(\check u^{-1}h\check u)$ are radially bounded, then
$h$ is also radially bounded.
\end{proposition}

\begin{proof}
If $\Pi_{[m,\infty)}(h)$ is radially bounded, then 
$$
\sup_{n\in\mathbb{N}}\frac{|\tr\left(\Pi_{[1,n]}(\Pi_{[m,\infty)}h)\right)|}{n}<+\infty,
$$
and, since for $n>m$ we have 
$$
\tr\left(\Pi_{[1,n]}(h)\right)=\tr\left(\Pi_{[1,n]}(\Pi_{[m,\infty)}h)\right)+\tr\left(\Pi_{[1,m]}(h)\right),
$$
it follows that 
$$
\sup_{n\in\mathbb{N}}\frac{|\tr\left(\Pi_{[1,n]}h\right)|}{n}<+\infty.
$$
It remains to show that 
$$
\sup_{n\in\mathbb{N}}\frac{\tr\left(\Pi_{[1,n]}h\right)^2}{n^2}<+\infty.
$$
Let $\pi$ be a permutation of $\mathbb N$ defined as follows:
\[
\pi(i) =
\begin{cases}
m+i &i=1,\ldots, m;\\
i-m &i=m+1, \ldots, 2m;\\
i &i>2m.
\end{cases}
\]
By definition, for any $h\in H$ we have
\[
(\check u^{-1}h\check u)_{ij} \;=\; \check h_{\pi(i)\pi(j)}.
\]
Consequently, for any $N \in \mathbb N$ we have
\[
\sum_{i,j=1}^N \left| h_{ij} \right|^2 \;\le\; \sum_{i,j=m+1}^N \left| h_{ij} \right|^2 + \sum_{i,j=m+1}^N \left|
(\check u^{-1}h\check u)_{ij} \right|^2 + \sum_{i,j=1}^{2m} \left| h_{ij} \right|^2.
\]

\end{proof}


Proposition \ref{radbdd} is now  immediate from Corollary \ref{projradbdd} and Propositions \ref{uinv},
\ref{radest}.

Theorem \ref{mainh} is proved completely.

\section{Proof of Theorem \ref{mainz}.}
The proof is similar (and simpler) in this case.
Again, a matrix $z\in\Mat$ will be called \emph{radially bounded} if
$$
\sup_{n\in \mathbb{N}}\frac{\tr\left(\Pi_{[1,n]}z\right)^{*}\left(\Pi_{[1,n]}z\right)}{n^2}<+\infty
$$
(here, as usual, the symbol $z^*$ stands for the transpose conjugate of a matrix $z$). As before, we assign to a matrix
$z\in\Mat$ the sequence $\mu_n^z$ of orbital measures corresponding to the sequence of compact subgroups
$U(n)\times U(n)$, $n\in\mathbb{N}$,
 and
say that a matrix $z\in\Mat$ is \emph{weakly recurrent} if for any bounded positive continuous function $f$ on
$\Mat$ we have
$$
\inf_{n\in\mathbb{N}}\int_{\Mat}f\,d\mu_n^z > 0
$$
Again
we have the following
\begin{proposition}
\label{radbddprecompz} If a  matrix $z\in \Mat$ is radially bounded then the sequence of orbital measures
$\mu_n^z$ is weakly precompact. In particular, if  $z$ is radially bounded, then $z$ is also weakly recurrent.
\end{proposition}
{\bf Remark.} As before, the converse statement also holds: if the sequence of orbital measures is weakly precompact, 
then $z$ is radially bounded.

\textbf{Proof.} This, again, follows from Rabaoui's work \cite{rabaoui1}, \cite{rabaoui2}. Indeed, let $z\in\Mat$, let
$$
z(n)=\Pi_{[1,n]}z,
$$
let
$$\lambda_1^{(n)}\geqslant\dots\geqslant\lambda_n^{(n)}\geqslant 0$$
be the eigenvalues of the matrix $(z(n))^{*}z(n)$ arranged in decreasing order, and set
$$
x_i^{(n)}=\frac{\lambda_i^{(n)}}{n^2},
\ \ \gamma^{(n)}=\frac{\tr\left(z(n)^{*}z(n)\right)}{n^2}.
$$
If $z$ is radially bounded, then any infinite set of natural numbers contains a subsequence $n_r$ such that the
sequence $\gamma^{(n^r)}$ as well as all the sequences $x_i^{n_r}$, $i=1,\dots,$ converge (to a finite limit) as
$r\to\infty$. In this case, by Rabaoui's theorem \cite{rabaoui1}, \cite{rabaoui2}, the sequence of orbital measures $\mu_{n_r}^z$
weakly converges to a probability measure as $r\to\infty$; weak precompactness is thus established.

To conclude the proof of the Theorem, it therefore remains to establish the following
\begin{proposition}
\label{radbddz} Let $m\in\mathbb{N}$ and let $\nu\in\classFNC$. Then $\nu$-almost every $z\in\Mat$ is radially
bounded.
\end{proposition}
The proof  follows the same pattern as that of Proposition \ref{radbdd}. Again, using Pickrell's classification of
ergodic probability measures as well as the ergodic decomposition theorem of \cite{Bufetov}, we have
\begin{proposition}
\label{radbddzfin} Let $\nu$ be a $U(\infty)\times U(\infty)$-invariant probability measure on $\Mat$. Then
$\nu$-almost every $z\in\Mat$ is radially bounded.
\end{proposition}

Given $\nu \in \mathfrak F(m, \Mat)$, we consider, again, the decomposition
\[
\nu \;=\; \int\limits_{{\rm Mat}(m, \mathbb C)} \nu_{z^{(m)}} \, d{\overline\nu} (z^{(m)}).
\]
Here ${\rm Mat}(m, \mathbb C)$ stands for the space of all $m\times m$-matrices with complex entries; the measure
$\overline\nu$ is the projection of $\nu$ onto ${\rm Mat}(m, \mathbb C)$ which is well-defined by definition of
the class $\mathfrak F(m, \Mat)$; and, for ${\overline \nu}$-almost every point $z^{(m)}\in {\rm Mat}(m, \mathbb
C)$ the measure $\nu_{z^{(m)}}$ is the canonical conditional probability measure given by Rohlin's Theorem. Again,
we have the following

\begin{proposition}
\label{zinv} If $\nu \in \mathfrak F(m, \Mat)$ is $U(\infty) \times U(\infty)$-invariant, then, for ${\overline
\nu}$-almost all $z^{(m)} \in {\rm Mat }(m, \mathbb C)$, the measure
\[
\textstyle \left( \Pi_{[m, \infty)} \right)_* \nu_{z^{(m)}}
\]
is also $U(\infty)\times U(\infty)$-invariant.
\end{proposition}
\begin{proof}
The proof of this Proposition is exactly the same as that of Proposition \ref{uinv}.
\end{proof}

It follows from Proposition \ref{zinv} that for $\nu$-almost every $z$, the matrix $\Pi_{[m, \infty)}z$ is radially bounded. 
To obtain boundedness for the matrix $z$ itself, we again apply a permutation of rows and columns.

Denote
\[
\textstyle \tau_n (z) \;=\;\displaystyle \tr\left( \left( \Pi_{[1,n]}z \right)^* \Pi_{[1,n]}z
\right) \;=\;  \sum_{i,j=1}^n \left| z_{ij} \right|^2.
\]

Let the matrix $\check u\in U(\infty)$ be defined by (\ref{uhatdef}).


The following clear inequality that holds for any $z \in \Mat$ and all  $n>3m$:
\[
\textstyle \tau_n (z) \;\le\; \tau_{2m}(z)+
\tau_n \left(\Pi_{[m, \infty)}z\right) + \tau_n \left(\Pi_{[m, \infty)}(
{\check u}^{-1} z{\check u}\right).
\]
Consequently, if $\nu \in \mathfrak F(m, \Mat)$ is $U(\infty) \times U(\infty)$-invariant, then 
 $\nu$-almost
every $z \in \Mat$ 
is radially bounded, and
Theorem \ref{mainz} is proved completely.

\end{document}